\def\Val{\mbox{\rm Val}}
\def\hoo{\mathfrak{h}}

\def\ZFC{\mbox{ZFC}}

\def\Mod{\mbox{Mod}}

\def\sol{second order logic}

\def\ma{\mathfrak{A}}
\def\mh{\mathfrak{H}}
\def\mb{\mathfrak{B}}
\def\mm{\mathfrak{M}}

\def\soc{second-order characterizable}

\def\qcat{quasi-categorical}
\def\wsoc{weakly second order characterizable}
\def\sock{$L^2_{\kappa\omega}$-characterizable}
\documentclass[11pt]{article}
\usepackage{amsthm,enumerate,comment} 
\usepackage[mathscr]{euscript}
\usepackage{amssymb,amsmath}
\usepackage{hyperref}
\usepackage{endnotes}
\newtheorem{theorem}{Theorem}
\newtheorem{definition}[theorem]{Definition}
\newtheorem{lemma}[theorem]{Lemma}
\newtheorem{corollary}[theorem]{Corollary}
\newtheorem{example}[theorem]{Example}

\def\P{\mathcal{P}}
\newcommand{\open}{\Bbb}

\newcommand{\oC}{{\open C}}

\newcommand{\oN}{{\open N}}
\newcommand{\oP}{{\open P}}
\newcommand{\oQ}{{\open Q}}
\newcommand{\oR}{{\open R}}

\begin{document}

\title{Model theory of second order logic\thanks{The  author would like to thank  the Academy of Finland, grant no: 322795, and funding from the European Research Council (ERC) under the
European Union’s Horizon 2020 research and innovation programme (grant agreement No
101020762). }}
\author{Jouko V\"a\"an\"anen\\
\\
University of Helsinki, Finland}
\date{}
\maketitle

\section{Introduction}

Identifying the model theory of second order logic may sound like an impossible mission. The common wisdom is that there is no such model theory.  Still, the final Open Problem in Chang-Keisler \cite{MR1059055} is: ``Develop model theory of second and higher order logic". 

It is clear that second order logic does not manifest model theory in the same sense as first order logic does.  In fact, much of currently popular first order model theory is based on the ability to eliminate first order quantifiers in certain situations. It is much more difficult, if not impossible, to eliminate second order quantifiers in any situation. Thus, if we are to find any model theory in second order logic, it is going to be model theory of a different kind.

It can be argued that first order model theory is totally based on the inability of first order logic to express things such as finiteness, well-foundedness, countability, and completeness of linear order. This inability has led to a rich structure theory which constitutes much of model theory---weakness has been turned into an asset.
Second order logic does not ``suffer" from such inability, and neither does it have  a rich structure theory, at least as far as can be seen today.

The most famous examples of second order axiomatizations are Dede\-kind's axiomatization of the arithmetic of natural numbers \cite{zbMATH02693454} and Hilbert's axiomatization of the arithmetic of real numbers \cite{Hilbert1900}. Both are categorical axiomatizations. Have they led to a better understanding of number theory or analysis? It is typical of textbooks of mathematics to give both axiomatizations as starting points. Arguably it has been beneficial in the development of mathematics to have these axiomatizations, but the actual progress in understanding the natural numbers or the real numbers is not based on methods of second order logic. In contrast, the study of \emph{first order} axiomatizations of natural numbers have led to insights such as G\"odel's Incompleteness Theorem and the Paris-Harrington Theorem.  In both cases the result is directly based on properties of first order logic.

What is then the use of second order model theory? As emphasised by Baldwin \cite{MR3230825}, what is interesting is not that there \emph{is} some categorical axiomatization for a structure\footnote{Our structures are the same as in first order model theory.} but that a particular natural axiomatization is in some sense categorical. This would seem to suggest that we should study particular theories rather than second order theories in general.

A special case is monadic \sol, by which we mean the fragment of second order logic which allows second order quantification over unary predicates only. This fragment is decidable on monadic structures \cite{zbMATH02615110}, on the successor structure $(\oN,s)$ \cite{MR0183636}, on the full infinite binary tree 
\cite{MR0231716}, on any countable ordinal $(\alpha,<)$ and on $(\omega_1,<)$ \cite{MR0476471}  (but $\ZFC$ cannot decide whether the monadic second order theory of $\omega_2$ is decidable \cite{MR704093}). In this paper we focus on non-monadic second order logic, even though the monadic case is highly interesting on particular structures. If the structure has something like a pairing-function, the monadic case encodes the non-monadic case and no advantage arises from restricting to it.

A well-known phenomenon in second order logic is the dependence on the background set theory in a way that modern model theorists usually find unsavoury. Indeed, the model theory of second order logic can be best described as set-theoretical model theory, that is, an investigation of properties of models and model classes using set-theoretical methods. Questions studied are related to definability, compactness, and cardinality. 

It is often thought and claimed that if we liberate second order logic from its entanglement with  set theory by admitting so-called Henkin models \cite{MR0036188}, we are squarely back in first order logic. The truth is different. Let us see in more detail why Henkin models are compared to first order logic. A Henkin model is a pair $(\ma,\mh)$, where $\ma$ is an ordinary model and $\mh$ is a set of subsets of the domain $A$ of $\ma$, relations on $A$ and functions on $A$. The Comprehension Axioms are assumed to be satisfied by $(\ma,\mh)$, meaning that at least all (second order) definable (with parameters) subsets, relations and functions are in $\mh$. In such definitions $\mh$ is used as the range of second order variables. We can interpret  a Henkin model $(\ma,\mh)$ as a two-sorted structure $\mb$ as follows. Let us assume, for the sake of simplicity, that $\mh$ consists only of (some) subsets of $A$. We can make $(\ma,\mh)$ a two-sorted structure in which sort $1$ is reserved for elements of $A$ and sort $2$ is reserved for the subsets in $\mh$. To create appropriate connection between the two sorts, a new binary relation $E$ is added such that $a E X$ holds exactly when $a\in A$, $X\in \mh$, and $a\in X$. Now, indeed, second order logic over $(\ma,\mh)$, i.e. second order logic over $\ma$ with second order variables ranging over $\mh$, can be reduced to first order logic over $\mb$. But $\mb$ is not just any first order structure. It has the binary predicate $E$ which codes a lot of combinatorial data. For example, $\mh$ necessarily contains all  finite subsets of $\ma$ because they are all definable (with parameters). Thus, if $\ma$ is infinite, $\mb$ is necessarily unstable. In fact, second order logic with Henkin models has a high degree of instability and is best compared to Peano arithmetic and ZFC set theory, rather than first order logic. In the light of this it should not come as a surprise that the model theory of second order logic looks at times like the model theory of set theory.

Suppose a model or a model class is definable in set theory. 
What more do we know about it if we know that it is definable in second order logic? The most dramatic difference concerns the structure $(\omega,+,\times)$: Assuming large cardinals, the second order theory of $(\omega,+,\times)$ is forcing absolute i.e. cannot be changed by forcing \cite{MR2723878}, something that is certainly  not true of set theoretic definability over this structure.

Questions of first order model theory that have, at least in the light of current knowledge, no relevance in second order logic, are: 

\begin{itemize}
\item Is the theory of a given model, e.g. $(\oQ,<)$, $(\oR,+,\times, e^x)$ or $(\oC,+,\times)$ decidable? (The second order theory of a structure is (rather trivially) decidable if and only if the structure is finite.) 

\item Is a given theory, such as dense linear order or algebraically closed fields,  decidable? (A second order theory is (rather trivially) decidable if and only if there is a natural number $n$ such that the theory has only models of  size $\le n$.) 
\end{itemize}

\noindent A question which is interesting in both first order and second order model theory, is:

\begin{itemize}
\item Given a complete theory, what is the structure of its models? (Naturally, complete second order theories have much \emph{fewer} models than complete first order theories.)
\end{itemize}

\noindent Genuinely second order model theory questions are the following:

\begin{itemize}
\item Which structures can be characterised up to isomorphism?
\item Which complete theories are categorical?
\item What are the Hanf-, L\"owenheim-, and compactness numbers of second order logic?
\item Under what conditions does a given model have a ``small" (not necessarily countable) elementary submodel?
\item Under what conditions does a given model have a ``big" (or even just a proper) elementary extension?
\item What is implicitly but not explicitly definable in second order logic?
\end{itemize}

This paper surveys answers to some (but not all) of these and  similar questions. It becomes now clear that in the case of second order logic the term ``model theory" gets a new meaning.  There are (some) general methods that can be used --- and are typical --- in the model theory of second order logic. These methods are by far not as powerful as the methods used in first order model theory. 

\emph{Notation:} Second order logic is denoted $L^2$. This logic has variables (and existential and universal quantifiers) for subsets, relations and functions on the domain of the model under consideration. If countable conjunctions and disjunctions are allowed, we denote the logic $L^2_{\omega_1\omega}$. Similarly $L^2_{\kappa\omega}$. For a general introduction to second order logic we refer to \cite{sep}.

\section{Second order characterizable structures}

In second order model theory categoricity is a dominating feature. Sentences with infinite models can be categorical, i.e. infinite models can be categorically characterized. Categoricity is indeed a typical feature in second order logic. Let us now introduce the relevant concepts:

\begin{definition}
A second order sentence\footnote{Or more generally a sentence of $L^2_{\kappa\omega}$.} is \emph{categorical}
if it has exactly one model up to isomorphism.
A model $\ma$ of a finite vocabulary is \emph{\soc} if it satisfies a categorical second order sentence $\theta_\ma$.
\end{definition}

For those accustomed to first order model theory this concept is shocking. In first order logic no infinite model is characterizable up to isomorphism, while all finite ones are. The truth is that in second order logic not only are there  \soc\ models but it is difficult to find one that is \emph{not} \soc.


Note the similarity with the Scott sentence of a countable model $\ma$ (\cite{MR0200133}). The difference is that our $\theta_\ma$ is a finite string of symbols while the Scott sentences are in the infinitary language $L_{\omega_1\omega}$. We could not have a ``Scott sentence" in second order logic for every countable model for cardinality reasons: There are continuum many non-isomorphic countable models, e.g. the models $(\oN,<,r)$, where $r\subseteq\oN$, but only countably many second order sentences. In fact, if $\Val^2$ is the set of G\"odel numbers of valid second order sentences, then $(\oN,<,\Val^2)$ is not \soc\ (see Theorem~\ref{oldmen} below).  An example of a non second order characterizable structure of a different kind is $(\kappa,<)$, where $\kappa$ is the first measurable cardinal $>\omega$
\cite{MR143710}.

If a model is \soc, then so are its reducts, which can be seen as follows: Suppose $\theta$ is a categorical sentence and $\ma\models\theta$. Let $\mb$ be the reduct of $\ma$ obtained by leaving out $R_1,\ldots, R_n$. Then $\mb$ satisfies the categorical sentence $\exists R_1\ldots\exists R_n\theta$. Note that this is not true in many-sorted vocabularies i.e. if the reduct is obtained by leaving out (also) some sorts, there is no way to say in second order logic that an expansion $\ma$ of the right kind exists. The reduct $\mb$ may be of the form $(\kappa,<)$, where $\kappa$ is the first measurable cardinal. Then $\mb$ is not \soc.  Still the many-sorted expansion $\ma$ may be \soc, for example it could be $(V_{\kappa+1},\in,\kappa,<)$.

Since reducts of \soc\ models are \soc, it makes sense to study the extreme case of  \soc\ models of the empty vocabulary. Then we talk about \emph{\soc\ cardinals} \cite{MR0416919}. If $\kappa$ is \soc, then 
so are $\kappa^+$ and $2^\kappa$. 
 In the opposite direction, it makes sense to try to find as rich \soc\ models as possible. 

The first and most famous example of a \soc\ model is $(\oN,s,0)$, where $s(n)=n+1$ \cite{zbMATH02693454}. This can be readily generalized. For example, if $f_1,\ldots, f_n$ are any recursive functions, then $$(\oN,f_1,\ldots,f_n)$$ is \soc. Another early and equally famous \soc\ model is  the ordered field of real numbers \cite{Hilbert1900} $$(\oR,+,\times,0,1).$$ Here, too, functions can be added to the structure as long as they are second order definable on the structure. Moreover, we can add the set of natural numbers and the resulting structure is \soc. In first order logic the field of reals is decidable, but if you add the natural numbers, the new structure is undecidable. This difference disappears in second order logic. The field of reals is exactly as undecidable in second order logic as the expansion by the set of natural numbers.

What would be an example of a model that is not \soc? By a cardinality argument we know that there must be a continuum of such models but it is not easy to put your finger on one. Constructions based on diagonalization, e.g the model $(\oN,<,\Val^2)$ (see Theorem~\ref{oldmen}), lead to such models  but what if we want a genuinely ``mathematical" example? This is a somewhat vague question. We give some examples in Section~\ref{non}, but only consistently.

A nice \emph{consequence} of second order characterizability of a model $\ma$ is the following: Suppose $\phi$ is a second order sentence in the vocabulary of $\ma$. Then $\ma\models\phi$ if and only if $\models\theta_\ma\to\phi$. Thus truth in $\ma$ is reduced to validity. Unfortunately second order logic has a Completeness Theorem only in so-called Henkin models. Thus, to use $\theta_\ma$ to derive the truth of $\phi$ in $\ma$, one has to check $\phi$ in all models of $\theta_\ma$, and alas, $\ma$ itself is one of those models. However, this is not as  circular as it appears to be. In `normal' cases $\theta_\ma\to\phi$ can be derived from the axioms of second order logic. If one does not like formal derivations, one can use informal semantic thinking and then appeal to the Henkin Completeness Theorem \cite{MR0036188}.

The second order theory of an infinite model is always undecidable because we can use second order quantifiers to ``guess" a copy of $(\oN,+,\times)$ inside the model (i.e. we use second order existential quantifiers to say that $\oN,+$ and $\times$ exist and are, up to isomorphism, what we usually mean by them), and then code the Halting Problem (or even $0^\sharp$ if it exists). This raises the question how complicated can the second order theory of a \soc\ structure be? Let us define for \soc\ models $\ma$ and $\mb$:
$$\ma\le_T\mb$$
if the second order theory of $\ma$ is Turing-reducible to the second order theory of $\mb$. Let us write $\ma\equiv_T\mb$ if $\ma\le_T\mb$ and $\mb\le_T\ma$.

\begin{theorem}[\cite{Vaananen2012-VNNSOL}]Suppose $\ma$ is \soc. 
\begin{enumerate}[(i)]
\item If $|A|\le|B|$,  then $\ma\le_T\mb$.
\item If $A$ is infinite and $\ma\le_T\mb$, then ${|A|}< 2^{|B|}$. 
\end{enumerate}
\end{theorem}  

\begin{proof}
(i) If $|A|\le|B|$, then $\ma\models\phi$ can be expressed on $\mb$ by guessing a subset of the universe as well as relations on it such that  the subset with the relations satisfy $\theta_\ma\wedge\phi$. (ii) If $2^{|B|}\le |A|$, we can guess not only a copy of $\mb$ on $A$ but also a truth-definition for second order logic on $\mb$. Then $\ma\nleq_T\mb$ follows by a standard undefinability of truth argument.
\end{proof}

\begin{corollary}
Assume GCH. Then for infinite $\ma$ and $\mb$, such that $\ma$ is \soc: $$\ma\le_T\mb\iff |A|\le|B|.$$
\end{corollary}

The theorem shows that \soc\ models of the same cardinality are Turing-equivalent. Thus the infinite \soc\ models form a hierarchy, where the countably many countable ones are lowest, then the countably many \soc\ models of cardinality $\aleph_1$, then of cardinality $\aleph_2$, etc. On each level all the models are Turing equivalent and Turing reducible to the models on a higher level. In general we cannot say whether the models on the higher level are Turing reducible to the models on the lower level, but if they are exponentially bigger, then they are not. 

Since there are only countably many \soc\ cardinals, it makes sense to ask how far they reach in the proper class of all cardinals. The following concept helps in estimating this:

\begin{definition}
The {\em L\"owenheim number} $\ell^2$ of \sol\ is the smallest cardinal $\kappa$  such that if a second order sentence has a model, then it has a model of cardinality $<\kappa$.
\end{definition}

The L\"owenheim-number $\ell^2$ exists  because it is the supremum over all second order sentences $\phi$, which have a model, of the least cardinality of a model of $\phi$. It is not a particularly small cardinal as it is a fixed point (of fixed points of fixed points etc) of the $\beth$-function, but it is not a so-called \emph{large cardinal} as it has cofinality $\omega$, and its existence is provable in ZFC. Using the method of \cite{Barwise1972-BARTHN} 
one can show that the existence of $\ell^2$ \emph{cannot} be proved from $\Sigma_1$-replacement alone as $V_{\ell^2}$ is a model of $\Sigma_1$-replacement, while it can be readily proved
from $\Sigma_2$-replacement.

We can use $\ell^2$ to compute the supremum of all \soc\ cardinals  i.e. the supremum of cardinalities of \soc\ models:

\begin{theorem}
The supremum of the sizes of all \soc\ models is $\ell^2$.
\end{theorem}

\begin{proof}
Firstly, $\ell^2$ clearly cannot be itself  \soc. Secondly, if  $\kappa$ is \soc\ then the second order sentence that characterises $\kappa$ must have a model $<\ell^2$, whence $\kappa<\ell^2$. Finally, suppose $\lambda<\ell^2$. We should find a \soc\ cardinal $\kappa$ such that $\lambda\le\kappa<\ell^2$. Since $\lambda<\ell^2$, there is a second order $\phi$ with a model but no models $<\lambda$. Let $\kappa$ be the least cardinality of a model of $\phi$. Clearly, $\lambda\le\kappa<\ell^2$. The cardinal $\kappa$ is \soc\ because it is the unique cardinal which on the one hand has relations which make $\phi$ true but on the other hand does not have a subset of smaller cardinality with relations which make $\phi$ true. 
\end{proof}


\section{Weakly second order characterizable structures}

Second order characterizability is, despite its ubiquitousness, too strong  to be of model theoretic interest---when you have just one model, up to isomorphism, it is hard to develop structure theory.   We introduce now a slightly weaker notion.
It turns out that second order characterizability of a structure is an amalgam of two things. One is the second order characterizability of the cardinality of the structure. The other is a weaker form of second order characterizability:

\begin{definition}
A second order sentence is  {\emph \qcat}
if it has at most one model in each cardinality, up to isomorphism. A model $\ma$ of a finite\footnote{We assume that the vocabulary is finite because a single sentence of second order logic can contain only finitely many symbols. It would be too much to ask that a single sentence fixes the interpretation of a non-logical symbol which does not even occur in the sentence.} vocabulary is \emph{\wsoc} if it satisfies a \qcat\ sentence. 
\end{definition}

In first order model theory quasi-categoricity is called \emph{total categoricity} when the theory or sentence has models in all infinite cardinalities. Moreover, in first order model theory there is Morley's Theorem: If a theory is $\kappa$-categorical for one uncountable $\kappa$ it is for all uncountable $\kappa$. This is of course not true in the second order case.

Every model of the empty vocabulary is \wsoc. Of course, a finite disjunction of categorical (or \qcat) sentences is \qcat. The most famous historical example of a \qcat\ sentence is the following:
 
\begin{example}[Zermelo \cite{zbMATH02562682}]
The conjunction of the finitely many axioms of the second order Zermelo-Fraenkel set theory $ZFC^2$ is \qcat. Its models are, up to isomorphism, of the form $(V_\kappa,\in)$, where $\kappa$ is strongly inaccessible. Thus all such models $(V_\kappa,\in)$ are \wsoc. Many of them are \soc, for example $(V_\kappa,\in)$, where $\kappa$ is the first (second, third, $\omega$'th, $\omega_1$'st, etc)  but not all, if there are uncountably many inaccessible cardinals.  
\end{example}

Note that the cardinality of a \soc\ model is always \soc\, but not conversely: Not all countable models are \soc\ although their cardinality is.

\begin{theorem}
A model is \soc\ if and only if it is \wsoc\ and its cardinality is \soc.
\end{theorem}

\begin{proof}
Suppose $\ma$ is \soc. A fortiori, $\ma$ is \wsoc\ and the reduct of $\ma$, namely the cardinality of $A$, is \soc. Conversely, suppose $\ma$ is \wsoc\ and $\theta$ is a \qcat\ sentence true in $\ma$. Suppose $\phi$ is a categorical sentence in the empty vocabulary, true in the reduct $|A|$ of $\ma$. Now $\theta\wedge\phi$ is categorical and true in $\ma$, whence $\ma$ is \soc.
\end{proof}

We have split the concept of second order characterizability of a model to two components. The first component, quasi-categoricity, is something that the model satisfies and that is categorical independently of the cardinality of the model. The second component is the ability of second order logic to capture the cardinality of the model. By separating these two components it becomes clearer why  this particular model is \soc.

Let us now look at more examples of \wsoc\ models:

\begin{example}
Free abelian groups are \wsoc, as the unique size of the free basis determines the isomorphism type and can be detected with second order logic. Many free abelian groups are also \soc\ but, for cardinality reasons, not all. A dense linear order without endpoints  is \wsoc\ if it is complete, as an easy transfinite back-and-forth argument shows. Any well-order of the form $(\kappa,<)$, $\kappa$ a cardinal, is \wsoc.
\end{example}

\begin{example}Every recursively axiomatizable first order theory with only infinite models is finitely axiomatizable in second order logic (\cite{MR0106175}). Every such theory which is $\kappa$-categorical for all uncountable $\kappa$ (or even for all infinite $\kappa$) is \qcat\ in second order logic. We get the following examples of \wsoc\ structures:
\begin{itemize}
\item Algebraically closed fields 
\item Infinite dimensional vector spaces over a fixed finite field
\item Divisible torsion-free abelian groups 
\end{itemize}
\end{example}

The class of models of a \qcat\ second order sentence have a little more potential for a structure theory than the categorical ones, for the models are separated from each other by their cardinality. In principle we can ask, how many models  are there  up to isomorphism? It is not much information but it is something. Set many? Class many?

\begin{definition}
A model is \emph{boundedly \wsoc\ }if it satisfies a \qcat\ sentence which does not have arbitrarily large models. In that case its \emph{height} is the smallest $\kappa$ such that for some \qcat\ sentence $\theta$ satisfied by $\ma$ has only models $\le\kappa$.
\end{definition}

\begin{example}
Every $(V_\kappa,\in)$, where $\kappa$ is strongly inaccessible but has only countably many strongly inaccessibles below, is boundedly \wsoc.
\end{example}

\begin{example}
  Many free abelian groups are boundedly \wsoc. For example, if there are measurable cardinals, the free abelian groups the size of which is less than the first measurable cardinal are boundedly \wsoc, because we can use second order logic to say that there are no measurable cardinals $\kappa$ such that $2^\kappa$ is below the size of the universe. Free abelian groups bigger than the first extendible cardinal are unboundedly \wsoc\ (\cite{MR0295904}).
\end{example}

We used the L\"owenheim number $\ell^2$ to put a bound on the cardinalities of \soc\ structures. We will now introduce another invariant for second order logic, and this one will put a bound on the cardinalities of boundedly \wsoc\ structures: 

\begin{definition}
The {\em Hanf number} $\hoo^2$ of \sol\ is the smallest cardinal $\kappa$  such that if a second order sentence has a model $\ge \kappa$, then it has arbitrarily large models.
\end{definition}

\begin{theorem}
The supremum of the cardinalities of boundedly \wsoc\ models is $\hoo^2$.
\end{theorem}

\begin{proof}
Suppose $\ma$ is boundedly \wsoc\ and $|A|=\kappa$. Let $\theta$ be a \qcat\ sentence, true in $\ma$, such that $\theta$ does not have arbitrarily large models.   Then $\theta$ cannot have a model $\ge\hoo^2$. Hence $\kappa<\hoo^2$. Suppose then $\kappa<\hoo^2$ is arbitrary. Let $\phi$ be a second order sentence which has a model of size $\ge \kappa$ but does not have arbitrarily large models.  Let $\theta$ be the sentence of the empty vocabulary which is true in a model of cardinality $\lambda$ if and only if there are models of $\phi$ of arbitrarily large size $<\lambda$ or of size $\lambda$. The sentence $\theta$ cannot have a model $\ge\hoo^2$. Hence $\theta$
is boundedly \wsoc.  The supremum of the cardinalities of models of $\theta$ is at least $\kappa$.\end{proof}

\section{Non second order characterizable structures}\label{non}

The most obvious reason why there are models that are not \soc\ is  cardinality: there are more non-isomorphic models than there are second order sentences. Another reason is the Axiom of Choice (AC) which implies the existence of objects, such as bases for vector spaces and well-ordering of arbitrary sets, without giving any definition for them. If $V=L$, then the global well-ordering of $L$ can be used to overcome this problematic feature of AC. If $V=L$ is violated, we get incidences of the failure of second order characterizability.

To this end, recall that a \emph{Hamel basis} for  $\oR$ is any set of reals which is a basis in  $\oR$ as  a $\mathbb{Q}$-vector space. The existence of a Hamel basis is a well-known consequence of AC.

\begin{theorem}[\cite{MR3116536}] Suppose $\oP$ is the usual forcing notion for adding one Cohen real. In the forcing extension by $\oP$ the model $(\oR, +, \times, 0, 1, B)$, where $B$ is any Hamel basis for  $\oR$, is not \wsoc.
\end{theorem}

\begin{proof} (Sketch)
Suppose the contrary: there is a Hamel basis $B$ such that $(\oR, +, \times, 0, 1, B)$ is \wsoc. It follows that there is a second order formula $\psi(x)$ such that 
$$x\in B\iff(\oR, +, \times, 0, 1)\models\psi(x).$$ 
Let $G$ be a $\mathbb{P}$-generic filter over the ground model $V$.  We define the real $r^G$ so that its binary expansion is 
$
r^G=0.{r_0}^G{r_1}^G{r_2}^G...$
where
${r_i}^G=(\bigcup G)(i).$
There is an equation which represents the real $r^G$  as a linear combination of finitely many elements of the basis $B$ and there is a finite condition which forces this equation. However, automorphisms of the forcing can be used to show that this is impossible.
 \end{proof}

\begin{theorem}[\cite{MR3116536}]\label{v2te}
Assume CH. If we add a Cohen subset to $\omega_1$, then in the forcing extension the following holds: If $T$ is a countable complete unstable first order theory then it has models $\ma$ and $\mb$ of size $\aleph_1$ such that $\ma$ and $\mb$ are $L_{\omega_1 \omega}^{2}$-equivalent but non-isomorphic.
\end{theorem}

Note that the model $\ma$ in Theorem~\ref{v2te} is necessarily non-$L^2_{\omega_1\omega}$-characterizable, because it cannot be distinguished from $\mb$ even with an $L^2_{\omega_1\omega}$-theory and not even just among models of cardinality $\aleph_1$. We cannot hope to get a similar result without forcing, provably in ZFC \cite{MR3001545}. 
It is worth noting that completeness and instability of a first-order theory are absolute properties in set theory.

With cardinality arithmetic assumptions we obtain, without forcing, a criterion for every model of a first order theory of size $\kappa$ being  \sock:

\begin{theorem}[\cite{MR3116536}]\label{main}
Let $T$ be a countable complete first-order theory.

\begin{enumerate}[(i)]
\item Suppose that $\kappa$ is a cardinal such that  $2^\lambda<2^\kappa$ for all $\lambda < \kappa$.
If every model of $T$ of size $\kappa$ is \sock\, then $T$ is a shallow, superstable theory without DOP or OTOP.

\item Suppose that  $\kappa$ is a regular cardinal such that $\kappa=\aleph_\alpha$, $\beth_{\omega_1}(|\alpha|+\omega) \leq \kappa$. 
If $T$ is a shallow, superstable theory without DOP or OTOP, then every model of $T$ of size $\kappa$ is \sock\ .   
\end{enumerate}
\end{theorem}  

For more results similar to Theorems~\ref{v2te} and \ref{main}, see \cite{MR3116536}.

A different kind of approach to limitations of second order definability is the following:
If $\phi$ is a second order sentence, we define
$$\Mod(\phi)=\{\mm: \mm\models\phi\}.$$ If $\phi$ characterizes a model $\ma$, then $\Mod(\phi)$ is just the class of models isomorphic to $\ma$. It is easy to see that then $\Mod(\phi)$ is $\Delta_2$ in set theory:
$$
\begin{array}{lcl}
\mb\in\Mod(\phi)&\iff & \mb\models\phi\\
\mb\notin\Mod(\phi)&\iff&\mb\models\neg\phi.
\end{array}$$ On the other hand, the set of G\"odel numbers of valid second order sentences is $\Pi_2$-complete (\cite{MR337481}), hence not $\Sigma_2$ and in particular, not $\Delta_2$. Contemplating this fact leads us to the following result:

\begin{theorem}[\cite{Vaananen2012-VNNSOL}]\label{old}
Second order validity is not second order definable on any second order characterizable structure.
\end{theorem}

Another consequence of such a contemplation is:

\begin{theorem}\label{oldmen}
The model $(\omega,+,\times,\Val^2)$, where $\Val^2$ is the set of G\"odel numbers of valid second order sentences, is not second order characterizable.
\end{theorem}

\section{Categoricity of second order theories}

Second order characterizability is concerned with the expressive power of individual second order sentences. We may ask the same question about second order theories. 
\begin{theorem}\cite{Ajtai}
Suppose $V=L$ and $T$ is a complete second order theory  in a finite vocabulary. If $T$ has a countable model, then $T$ is categorical.
\end{theorem}

\begin{proof} Without loss of generality, the finite vocabulary of $T$ is $\{R\}$, where $R$ is a binary predicate. Since $V=L$, there is a   $\Sigma^1_2$-well-order $\prec$  of $\P(\omega^2)$.
Suppose $\ma$ and $\ma'$ are countable models of $T$. Since $T$ is complete, these two models satisfy the same second order sentences. Without loss of generality, both have $\omega$ as their domain and if $\ma=(\omega,S)$, then $S$ is $\prec$-minimal as an element of $\P(\omega^2)$, among isomorphic models, and similarly for $\ma'$. We show that $\ma=\ma'$ i.e. $R^\ma=R^{\ma'}$. To this end, let $(m,n)\in R^\ma$. We can express $(m,n)\in R^\ma$ in second order logic as follows. Let $\phi_{m,n}(R,N,<)$ be a second order sentence which says that $(N,<)\cong(\omega,<)$ and if $S\subseteq N\times N$ is the $\prec$-least $S$ such that $(N,S)\cong(\omega,R)$, then $(m,n)\in S$.  How to say ``$\prec$-least"? The sentence $\phi_{m,n}(R,N,<)$ determines unique arithmetic operations $+$ and $\times$ on the  $(N,<)$ of its models. Using these arithmetic operations we can write a second order formula $\theta(X,Y)$ with two free second order variables $X,Y$, which says $X,Y\subseteq N$ and $X\prec Y$. We can use $\theta(X,Y)$ to express ``$\prec$-least". The sentence $\exists N\exists\!<\!\phi_{m,n}(R,N,<)$ is true in $\ma$ and hence in $\ma'$. It follows that $(m,n)\in {R}^{\ma'}$.
\end{proof}

\begin{lemma} If $2^{\omega}<2^{\kappa}$, then there is a complete second order theory in a finite vocabulary with a model of cardinality $\kappa$ such that $T$ is not categorical.
\end{lemma}

\begin{proof} Let the vocabulary consist of a binary predicate symbol $<$ and a unary predicate symbol $P$.
There are only $2^{\omega}$ different second order theories in vocabulary $\{<,P\}$, but there are $2^\kappa$ non-isomorphic models $(\kappa,<,A)$, $A\subseteq \kappa$. Thus there are non-isomorphic $(\kappa,<,A)$ and $(\kappa,<',A')$ such that the second order theory of $(\kappa,<,A)$ is the same as the second order theory of $(\kappa,<',A')$.
\end{proof}

\begin{theorem}\cite{Ajtai} Assume $V=L$ and $\oP$ is Cohen forcing for adding a new real. Suppose $G$ is $\oP$ generic over $V$. Then there is a second order theory $T$ in a finite vocabulary such that in $V[G]$ the theory $T$ is complete, with a countable model, but non-categorical.
\end{theorem}

\begin{proof} For any $X\subseteq \omega$ let $F_X=\{Y\subseteq\omega :
|Y\triangle X|<\omega\}$ and let $\ma_X$ be the model $(\omega\cup F_X, <, R)$, where $<$ is the natural order of $\omega$ and $R$ is the binary relation $\{(n,Y) : n\in Y, Y\in F_X\}$. Without loss of generality, $G$ is a subset of $\omega$. Note that $\omega\setminus G$ is also in $V[G]$ and is $\oP$-generic over $V$. Clearly $\ma_G\not\cong \ma_{\omega\setminus G}$. However,  we now show that $\ma_G$ and $\ma_{\omega\setminus G}$ are second oder equivalent. Suppose $p$ forces a second oder sentence $\phi$ to be true in $\ma_G$. Let $G'$ be a finite modification of $G$ such that $p\in \omega\setminus G'$. Note that $\omega\setminus G'$ is $\oP$-generic over $V$. Hence $\phi$ is true in $\ma_{\omega\setminus G'}$. But $\ma_{\omega\setminus G}=\ma_{\omega\setminus G'}$. So $\phi$ is true in $\ma_{\omega\setminus G}$, as desired.

\end{proof}

As the above results demonstrate the question of categoricity of complete second order theories is highly non-trivial. For more results in this direction we refer to \cite{Solo} and \cite{VaaWoo}.


\section{What is left out?}

The following topics clearly belong to the general area of the model theory of second order logic, in addition to the topics treated in this paper:

\begin{itemize}
\item {\bf Model theory of Henkin models:}  If second order logic is endowed with the Henkin semantics, it resembles first order logic. However, to give the second order quantifiers some non-trivial power, and to thereby distinguish second order logic from two-sorted first order logic, one usually assumes the Comprehension Axioms of \cite{zbMATH03029084}.  The effect of the Comprehension Axioms is that any non-trivial theory becomes unstable. See e.g. \cite{MR3116536} for concrete results manifesting this.
\item {\bf Second order logic and  large cardinals}. Second order logic has a special relationship with measurable, supercompact and extendible cardinals, see \cite{MR143710} and \cite{MR0295904}. These large cardinals play the same role in second order model theory as $\aleph_1$ plays in first order model theory.
\item {\bf Model theory of particular logics:} The model theory of several sublogics of second order logic have been studied, e.g. the generalized quantifier $Q_1$ \cite{MR263616}, the cofinality quantifier $Q^{cf}_\omega$ \cite{MR376334}, the Magidor-Malitz quantifier $Q^{MM}_1$ \cite{MR453484}, and the stationary logic $L(aa)$ \cite{MR486629,MR622785}. 
\item {\bf First order model theory of structures with  second order properties:} Rather than studying the model theory of full second order logic, or the model theory of an entire sublogic of second order logic, one may study first order model theory of particular second order sentences, see e.g. \cite{MR501097,MR501098,MR506531,MR625527,MR725733}.

\end{itemize}


\end{document}